\documentclass[a4paper]{article}
\usepackage[dvips]{graphics}
\usepackage[english]{babel}
\usepackage{amsfonts}
\usepackage{amsthm}
\usepackage{amsmath}
\usepackage{amscd}
\usepackage[active]{srcltx} 
\usepackage{latexsym}
\usepackage{amssymb}
\usepackage[latin1]{inputenc}
\usepackage{xcolor}
\usepackage{enumitem}

\numberwithin{equation}{section}

\newtheorem{thm}{Theorem}[section]
\newtheorem{prop}[thm]{Proposition}
\newtheorem{innercustomprop}{Proposition}
\newenvironment{customprop}[1]
  {\renewcommand\theinnercustomprop{#1}\innercustomprop}
  {\endinnercustomprop}
\newtheorem{cor}[thm]{Corollary}
\newtheorem*{cor*}{Corollary}
\newtheorem{lema}[thm]{Lemma}
\newtheorem*{lema*}{Lemma}

\newtheorem{hyp2}[thm]{Hypothesis}

\theoremstyle{definition}

\newtheorem{prob}{Problem}
\newtheorem*{prob*}{Problem}
\newtheorem{innercustomthm}{Problem}
\newenvironment{customprob}[1]
  {\renewcommand\theinnercustomthm{#1}\innercustomthm}
  {\endinnercustomthm}

\newtheorem{obs}[thm]{Remark}

\newtheorem*{obs*}{Remark}
\newtheorem*{thm*}{Theorem}
\newtheorem*{prop*}{Proposition}

\newcommand{\PI}[2]{\left\langle \,#1 , #2\, \right\rangle}

\newcommand{\set}[1]{\left\{ \,#1\, \right\}}

\newcommand{\parentesis}[1]{\left( \,#1\, \right)}
\newcommand{\parentesisb}[1]{\big( \,#1\, \big)}

\newcommand{\w}[1]{\widetilde{#1}}

\newcommand{\x}{\times}

\newcommand{\RR}{\mathbb{R}}

\newcommand{\St}{\mathcal{S}}

\newcommand{\N}{\mathcal{N}}

\newcommand{\HH}{\mathcal{H}}

\newcommand{\ZZ}{\mathcal{Z}}
\newcommand{\mc}[1]{\mathcal{#1}}

\newcommand{\noi}{\noindent}

\newcommand{\dx}{\Delta x}

\newcommand{\la}{\lambda}

\makeatletter
\newcommand{\eqnum}{\refstepcounter{equation}\textup{\tagform@{\theequation}}}
\makeatother

\DeclareMathOperator{\real}{Re}

\begin{document}

\title{Quadratic programming with one quadratic constraint in Hilbert spaces}

\author{Santiago Gonzalez Zerbo \and Alejandra Maestripieri \and Francisco Mart\'{\i}nez Per\'{\i}a}

\date{}

\maketitle

\begin{abstract}
A quadratically constrained quadratic programming problem is considered in a Hilbert space setting, where neither the objective nor the constraint are convex functions. Necessary and sufficient conditions are provided to guarantee that the problem admits solutions for every initial data (in an adequate set).
\end{abstract}

\section{Introduction}

As the simplest form of non-linear programming, {\it quadratic programming} (QP) plays a significant role in optimization.
Portfolio optimization, (determination of) economic equilibria, control theory and machine learning
\cite{McCarl,Gupta,Loosli}, are all areas that are naturally approached via QP.
In particular, {\it quadratically constrained quadratic programming} (QCQP), which we consider here, appears in the context of
the steering direction estimation for radar detection, the maximum cut problem and boolean optimization \cite{Golub,DeMaio,Boyd}, among others.
QCQP problems have been extensively studied, particularly in the finite dimensional setting
\cite{Powell,Ye,Polik,Park}.
In the infinite dimensional setting, these problems were first analyzed in \cite{Kelly,Benson,Shimizu,Pawel,Semple} and more recently
algorithms have been developed for particular cases \cite{Ahmetoglu,Chen}.

In Hilbert spaces, QCQP can be posed as
\begin{alignat*}{3}
& \text{minimize}   \quad && f(x)&&=\PI{T_0x}{x}+2\real\PI{y_0}{x} + \alpha_0\\
& \text{subject to}   \quad && g_i(x)&&=\PI{T_i x}{x}+2\real\PI{y_i}{x}\leq\alpha_i,\qquad i=1,2,...,m,
\end{alignat*}
where the optimization variable $x$ is in a Hilbert space $(\HH,\PI{\cdot}{\cdot})$, and the data consists of bounded selfadjoint operators
$T_i$ acting in $\HH$, vectors $y_i\in\HH$ and scalars $\alpha_i\in\RR$, for $i=0,1,...,m$. Such problems have been recently analyzed in \cite{Bednarczuk}
for real Hilbert spaces assuming the existence of solutions, and were studied in \cite{Dong1,Dong2}
assuming the operators $T_i$ were positive semidefinite, leading to convex constraints.
These cases can be analyzed with classical convex optimization techniques \cite{Rockafellar,Sundaram,Boltyanski,Boyd}.

The case where only one constraint is considered (QP1QC) is related to practical issues derived from system identification and machine learning theory.
In the finite dimensional setting, QP1QC occur in the time of arrival geolocation problem \cite{Hmam}, and in time series model identification
under noisy measurements \cite{Pelt_1,Pelt_2,Pelt_3}. In classical machine learning theory,
usually formulated in reproducing kernel Hilbert spaces (RKHS) \cite{Signoretto}, the objective functions in the QCQP are convex due to the fact that the
kernel is positive definite \cite{Gartner,Xu,Mohri}. However, verifying that the kernel is positive definite (known as the Mercer condition) is
computationally very hard. In \cite{Canu2004,Canusplines,Oglic,Oglic_2} it has been proposed to
use reproducing kernel Krein spaces (RKKS) where the kernel is indefinite, thus avoiding the
need to verify the Mercer condition. Indefinite kernel techniques
have also proved useful in pattern recognition theory \cite{Haasdonk2005,Sonnenberg2006}.

\medskip

This paper is devoted to studying the following QP1QC:

\begin{prob}\label{pb 1}
Given $A,B$ bounded selfadjoint operators acting in a Hilbert space $(\HH,\PI{\cdot}{\cdot})$, vectors $a,b\in\HH$, and a constant $\beta\in\RR$,
analyze the existence of
\[
\min\,\PI{Ax}{x}+2\real\PI{a}{x}\quad\text{subject to}\quad\PI{Bx}{x}+2\real\PI{b}{x}\leq\beta,
\]
\end{prob}
\noi and the associated {\it quadratic programming problem with one equality quadratic constraint} (QP1EQC):

\begin{prob}\label{pb 2}
Analyze the existence of
\[
\min\,\PI{Ax}{x}+2\real\PI{a}{x}\quad\text{subject to}\quad\PI{Bx}{x}+2\real\PI{b}{x}=\beta.
\]
\end{prob}

\medskip

We assume that $A$ and $B$ are indefinite (neither positive nor negative semidefinite), so the objective function and the constraint in Problem \ref{pb 2}
are not convex. This gives rise to significant difficulties since the problem is not amenable to classical convex optimization techniques,
though in the finite dimensional setting its analogous version has been shown to be polynomially solvable \cite{Boyd,Polik,HLS14}.

\medskip

The main result of this paper provides necessary and sufficient conditions for Problem \ref{pb 1} to admit a solution for every initial data point
in $\big(R(A)+R(B)\big)\x R(B)$.
A necessary condition for the existence of solutions to Problem \ref{pb 1} is that there exists $\la_0\in\RR$ such that $A+\la_0 B$ is positive
semidefinite. In this case, there exists a closed interval $[\la_-,\la_+]$ such that the operator pencil $A+\la B$ is positive semidefinite for
$\la$ in $[\la_-,\la_+]$. We first show that also $\la_->0$ is a necessary condition for Problem \ref{pb 1} to admit a
solution, and (under this hypothesis) Problem \ref{pb 1} admits a solution if and only if
Problem \ref{pb 2} does. Then the sets of solutions to both problems are the same, and we can focus on Problem \ref{pb 2}.

\medskip

In what follows, Section \ref{section_pencils} introduces the notation and presents some results on linear operator pencils
which are used in the rest of the paper. Given bounded selfadjoint operators $A$ and $B$,
a well known result by Krein and Smul'jan \cite{KS} states that if $B$ is indefinite then there exists $\la_0\in\RR$ such that
$A+\la_0 B$ is positive semidefinite if and only if $\PI{Ax}{x}\geq0$ whenever $\PI{Bx}{x}=0$.
Moreover, $A+\la B$ is positive semidefinite for every $\la\in[\la_-,\la_+]$, where
\[
\la_-=-\inf_{\{x\in\HH:\PI{Bx}{x}>0\}}\frac{\PI{Ax}{x}}{\PI{Bx}{x}}\quad\quad\text{and}\quad\quad\la_+=-\sup_{\{x\in\HH:\PI{Bx}{x}<0\}}\frac{\PI{Ax}{x}}{\PI{Bx}{x}}.
\]
Section \ref{section_problem} presents necessary and sufficient conditions for the existence of solutions to Problem \ref{pb 2} for a fixed
constant $\beta\in\RR$ and initial data $(a,b)\in\HH\x\HH$.

Section \ref{section_conditions} is devoted to finding a set of necessary and sufficient conditions for the existence of solutions to Problem
\ref{pb 2} for every initial data point in $\big(R(A)+R(B)\big)\x R(B)$.
We reduce the initial problem to an equivalent simpler problem.
First, it is necessary that $\la_-<\la_+$ with values obtained at
\[
\la_-=-\min_{\{x\in\HH:\PI{Bx}{x}>0\}}\frac{\PI{Ax}{x}}{\PI{Bx}{x}}\quad\quad\text{and}\quad\quad\la_+=-\max_{\{x\in\HH:\PI{Bx}{x}<0\}}\frac{\PI{Ax}{x}}{\PI{Bx}{x}}.
\]
Then it is necessary to have $R(A+\la B)=R(A)+R(B)$ for every $\la\in(\la_-,\la_+)$, and that an analogous condition is satisfied
for the case in which $\la=\la_-$ or $\la=\la_+$.
The necessary conditions are also shown to be sufficient. Necessary and sufficient conditions for
Problem \ref{pb 2} to admit a solution for every $(a,b)\in\HH\x R(B)$ are also established.

\section{Some results on linear operator pencils with a range of positiveness}\label{section_pencils}

In this work $\HH$ denotes a complex separable Hilbert space, and $\mc{L}(\HH)$ stands for the algebra of bounded linear operators in $\HH$.




An operator $A\in \mc{L}(\HH)$ is {\it positive semidefinite} if $\PI{Ax}{x}\geq 0$ for all $x\in\HH$; and it is {\it positive definite} if there exists
$\alpha>0$ such that $\PI{Ax}{x}\geq \alpha\|x\|^2$ for every $x\in\HH$. The cone of positive semidefinite operators is denoted by $\mc{L}(\HH)^+$.
If $A,B\in \mc{L}(\HH)$ are selfadjoint, $A\geq B$ stands for $A-B\in\mc{L}(\HH)^+$.
Every selfadjoint operator $A\in\mc{L}(\HH)$ admits a canonical decomposition as the difference of two positive operators:
there exist unique closed subspaces $\HH_+$ and $\HH_-$ of $\HH$, and injective operators $A_+\in\mc{L}(\HH_+)^+$ and $A_-\in\mc{L}(\HH_-)^+$ such that 
\begin{equation}\label{eq:descomposicion}
\HH=\HH_+\oplus\HH_-\oplus N(A),
\end{equation}
and $A=\left(\begin{smallmatrix}
A_+ & 0 & 0 \\
0 & -A_- & 0 \\
0 & 0 & 0
\end{smallmatrix}\right)$ with respect to \eqref{eq:descomposicion}.

We say that a selfadjoint operator $A\in\mc{L}(\HH)$ is {\it indefinite} if it is neither positive nor negative semidefinite;
i.e., if $A_+$ and $A_-$ are both non-zero.
If $A\in\mc{L}(\HH)$ is selfadjoint, define the sets
\[
Q(A):=\set{x\in\HH\,:\,\PI{Ax}{x}=0},
\]
\[
\mc{P}^-(A):=\set{x\in\HH\,:\,\PI{Ax}{x}<0}\quad\text{and}\quad\mc{P}^+(A):=\set{x\in\HH\,:\,\PI{Ax}{x}>0}.
\]

\smallskip

In this section we consider linear pencils of the form $P(\la)=A + \la B$, where $A,B\in\mc{L}(\HH)$ are selfadjoint,
and the parameter $\la\in\RR$. Let
\begin{align*}
I_\geq (A,B) &:=\{\,\la\in\RR:\ A + \la B\in\mc{L}(\HH)^+\,\}, \\
I_> (A,B) &:=\{\,\la\in\RR:\ A + \la B\,\,\text{is positive definite}\,\}.
\end{align*}

The following statement is a version of a well known result by Krein-Smul'jan \cite[Thm. 1.1]{KS}, see also \cite[Lemma 1.35]{Azizov}. It characterizes the case in which $I_{\geq}(A,B)$ is non empty,
and also describes this set as a closed interval.

\begin{prop}\label{Azizov pencils}
If $B$ is indefinite, then $I_\geq (A,B)\neq\varnothing$ if and only if 
\[
\PI{Ax}{x}\geq 0 \qquad \text{whenever}\qquad \PI{Bx}{x}=0. 
\]
In this case, if
\begin{equation}\label{eq:def_lambdas}
\la_-:=-\inf_{x\in\mc{P}^+(B)}\frac{\PI{Ax}{x}}{\PI{Bx}{x}}\quad\quad\text{and}\quad\quad \la_+:=-\sup_{x\in\mc{P}^-(B)}\frac{\PI{Ax}{x}}{\PI{Bx}{x}},
\end{equation}
then $\lambda_-\leq\lambda_+$ and $I_\geq (A,B)=[\lambda_-,\lambda_+]$.
\end{prop}

When both $A$ and $B$ are indefinite with $I_\geq(A,B)\neq\varnothing$, then either $\la_->0$ or $\la_+<0$. Moreover,

\begin{prop}\label{prop_nueva_pencils} If $A$ and $B$ are indefinite, the following conditions are equivalent:
\begin{enumerate}[label=\roman*)]
\item $\PI{Ax}{x}\geq0$ whenever $\PI{Bx}{x}\leq0$;
\item $I_\geq(A,B)\neq\varnothing$ and $\la_->0$;
\item $I_\geq(A,B)\neq\varnothing$ and $\PI{Ax}{x}>0$ whenever $\PI{Bx}{x}<0$.
\end{enumerate}
\end{prop}

\begin{proof}
\noi{\it i)$\to$ii)} If {\it i)} holds, in particular $\PI{Ax}{x}\geq0$ whenever $\PI{Bx}{x}=0$ and, by Proposition \ref{Azizov pencils}, $I_\geq(A,B)\neq\varnothing$.
Also,
\[
\la_+=-\sup_{x\in\mc{P}^{-}(B)}\frac{\PI{Ax}{x}}{\PI{Bx}{x}}\geq0,
\]
and since $A$ is indefinite, $0\notin[\la_-,\la_+]$. Thus, $\la_->0$.

\smallskip

\noi{\it ii)$\to$iii)} If $I_\geq(A,B)\neq\varnothing$ and $\la_->0$, consider $x\in\HH$ such that $\PI{Bx}{x}<0$. From $\PI{(A+\la_- B)x}{x}\geq0$
we get that $\PI{Ax}{x}\geq-\la_-\PI{Bx}{x}>0$.

\smallskip

\noi{\it iii)$\to$i)} Since $I_\geq(A,B)\neq\varnothing$, $\PI{Ax}{x}\geq0$ whenever $\PI{Bx}{x}=0$. Then {\it i)} follows.
\end{proof}

Replacing $B$ by $-B$ we also get

\begin{cor}\label{} If $A$ and $B$ are indefinite, the following conditions are equivalent:
\begin{enumerate}[label=\roman*)]
\item $\PI{Ax}{x}\geq0$ whenever $\PI{Bx}{x}\geq0$;
\item $I_\geq(A,B)\neq\varnothing$ and $\la_+<0$;
\item $I_\geq(A,B)\neq\varnothing$ and $\PI{Ax}{x}>0$ whenever $\PI{Bx}{x}>0$.
\end{enumerate}
\end{cor}

From now on we assume that $B$ is indefinite and that $I_\geq(A,B)\neq\varnothing$.
The following results can be found in \cite{GZMMP2}.
If $\la_-<\la_+$ the ranges of $(A+\la B)^{1/2}$ are constant for $\la\in(\la_-,\la_+)$, implying that $N(A+\la B)$ are also constant (and equal to
$N(A)\cap N(B)$), while $N(A+\la_- B)$ and $N(A+\la_+B)$ contain $N(A)\cap N(B)$.
If the supremum (respectively, the infimum) is attained in \eqref{eq:def_lambdas}, then $N(A+\la_-B)$ (respectively, $N(A+\la_+B)$) strictly contains
$N(A)\cap N(B)$.

\begin{prop}\label{nucleo del interior}
Assume that $\lambda_-<\lambda_+$. Then,
\begin{enumerate}[label=\roman*)]
\item $N(A + \la B)=Q(A)\cap Q(B)=N(A)\cap N(B)$, \quad for every $\la\in (\lambda_-,\lambda_+)$;
\item $N(A+\la_\pm B)=\displaystyle{\set{x\in\mc{P}^{\mp}(B)\,\,:\,\,\frac{\PI{Ax}{x}}{\PI{Bx}{x}}=
-\la_\pm}}\cup\big(N(A)\cap N(B)\big).$
\end{enumerate}
\end{prop}

\begin{prop}\label{rango raiz}
Assume thath $\lambda_-<\lambda_+$. Then,
\begin{enumerate}[label=\roman*)]
\item $R\big((A + \la B)^{1/2}\big)=R\big((A + \la' B)^{1/2}\big)$ \quad for every $\la,\la'\in (\la_-,\la_+)$; 
\item $R\big((A + \la_\pm B)^{1/2}\big)\subseteq R\big((A + \la B)^{1/2}\big)$ \quad for every $\la\in (\la_-,\la_+)$. 
\end{enumerate}
\end{prop}

\smallskip

In the following, using Proposition \ref{rango raiz} we show how the study of the behaviour of pencil $A+\la B$ can be reduced to studying the behaviour of
an auxiliary pencil of the form
$I+\eta G$ acting on a dense subspace of $\HH$, where $G$ is some selfadjoint operator.
This reduction technique is widely used in the operator pencils context,
considering that the auxiliary pencil is easier to analyze, see e.g. \cite{Gohberg}.

Since $A+\la B\equiv 0$ in $N(A)\cap N(B)$ (and it is an invariant subspace for the pencil), from now on we assume that $N(A)\cap N(B)=\set{0}$.
Whenever $\la_-<\la_+$, set $\rho:=\frac{\la_-+\la_+}{2}$ and write
\begin{equation}\label{eq:def_W}
W:=(A+\rho B)^{1/2}\quad\quad\text{and}\quad\quad \mc{R}:=R(W).
\end{equation}
By Propositions \ref{nucleo del interior} and \ref{rango raiz}, $W$ is injective and $\mc{R}$ is dense. Moreover,
\[
\mc{R}=R\big((A+\la B)^{1/2}\big)\quad\quad\text{for every $\la\in(\la_-,\la_+)$.}
\]
For $\la\in[\la_-,\la_+]$, let $D=D(\la)\in\mc{L}(\HH)$ be the (unique) solution to $(A+\la B)^{1/2}=WX$
given by Douglas' Lemma \cite{Douglas}. It follows that
\begin{equation}\label{eq:Douglas}
A+\la B=WDD^*W\quad\quad\text{for every $\la\in[\la_-,\la_+]$,}
\end{equation}
and if $\la\in(\la_-,\la_+)$, then $D=W^{-1}(A+\la B)^{1/2}$ is invertible.

\begin{lema}\label{lema_G_acotado} Assume that $\la_-<\la_+$. Then, there exists a selfadjoint operator
$G\in\mc{L}(\HH)$ such that
\[
B=WGW.
\]
\end{lema}

\begin{proof}
Given $\la\in(\la_-,\la_+)$, $\la\neq\rho$, since $A+\la B=A+\rho B+(\la-\rho)B=W^2+(\la-\rho)B$,
by \eqref{eq:Douglas} we have that $WDD^*W=W^2+(\la-\rho)B,$
or equivalently, $(\la-\rho)B=W(DD^*-I)W$. Then,
\[
B=\frac{1}{\la-\rho}W(DD^*-I)W=WGW,
\]
where $G:=\frac{1}{\la-\rho}(DD^*-I)\in\mc{L}(\HH)$ is selfadjoint. Since $W$ is injective, we get that
\begin{equation*}\label{eq:definiendo_G}
G\big|_{\mc{R}}=W^{-1} BW^{-1}\big|_\mc{R},
\end{equation*}
and therefore $G$ only depends on $\rho$.
\end{proof}

Denote by $G\in\mc{L}(\HH)$ the (unique) selfadjoint operator
given in Lemma \ref{lema_G_acotado}. Then, $DD^*=I+(\la-\rho)G$ is positive semidefinite for $\la\in[\la_-,\la_+]$ and
\begin{equation}\label{eq:pencil_alternativo}
A+\la B=W\big(I+(\la-\rho) G\big)W\quad\quad\text{for each $\la\in[\la_-,\la_+]$.}
\end{equation}

Thus,

\begin{cor}\label{cor:lambdas_distintos_0} 	If $\la_-<\la_+$ then,
\[
R(A)+R(B)=\text{\normalfont span}\bigg\{\bigcup_{\la\in [\la_-,\la_+]}R(A+\la B)\bigg\}\subseteq \mc{R}.
\]
\end{cor}

We now analyze the case in which $R(A+\rho B)=R(A)+R(B)$. We frequently use that for $C,D\in\mc{L}(\HH)$,
\[
R(C+D)=R(C)+R(D)\quad\quad\text{if and only if}\quad\quad R(C)\subseteq R(C+D),
\]

\begin{prop}\label{prop:rango_invariante}
Assume that $\la_-<\la_+$. Given $\la\in(\la_-,\la_+)$:
\begin{enumerate}[label=\roman*)]
\item if $\la\neq\rho$, $R(A+\la B)\subseteq R(A+\rho B)$ if and only if $R(A+\rho B)=R(A)+R(B)$;
\item $R(A+\la B)=R(A)+R(B)$ if and only if $\big(I+(\la-\rho)G\big)(\mc{R})=\mc{R}+G(\mc{R})$;
\item $R(A+\la B)=R(A+\rho B)$ if and only if $\big(I+(\la-\rho)G\big)(\mc{R})=\mc{R}$.
\end{enumerate}
\end{prop}

\begin{proof} 
\noi{\it i)} If $R(A+\la B)\subseteq R(A+\rho B)$, then
\[
(\la-\rho)Bx=(A+\la B)x-(A+\rho B)x\in R(A+\rho B)\quad\quad\text{for every $x\in\HH$.}
\]
Hence, $R(B)\subseteq R(A+\rho B)$, or equivalently, $R(A+\rho B)=R(A)+R(B)$. The converse is trivial.

\smallskip

\noi{\it ii)} Since $R(A+\la B)=W\big(I+(\la-\rho)G\big)(\mc{R})$ and
\[
R(A)+R(B)=R(W^2)+R(B)=W\big(\mc{R}+G(\mc{R})\big),
\]
applying $W^{-1}$ on both sides of $R(A+\la B)=R(A)+R(B)$ yields 
$\big(I+(\la-\rho)G\big)(\mc{R})=\mc{R}+G(\mc{R})$. The converse follows again applying $W$ to both sides of the equality.

\smallskip

\noi{\it iii)} By \eqref{eq:pencil_alternativo}, $R(A+\la B)=R(A+\rho B)$ if and only if $W\big(I+(\la-\rho)G\big)(\mc{R})=W(\mc{R})$, or equivalently,
applying $W^{-1}$ on both sides of this last equality, $\big(I+(\la-\rho)G\big)(\mc{R})=\mc{R}$.
\end{proof}

Since $I+(\la-\rho)G$ is invertible for any $\la\in(\la_-,\la_+)$, the equality $\big(I+(\la-\rho) G\big)(\mc{R})=\mc{R}$ can be trivially expressed as 
\[
\big(I+(\la-\rho) G\big)(\mc{R})=\mc{R}\cap R\big(I+(\la-\rho) G\big).
\]
The next lemma states a more general result that also considers the case in which $\lambda$ takes the values in the extrema of the
interval. This is a technical result required in the last section.

\begin{lema}\label{lema_condiciones_rangos} Assume that $\la_-<\la_+$. Given $\la\in[\la_-,\la_+]$, then
\begin{equation}\label{eq_dos_condiciones}
R(A+\la B)=R(A+\rho B)\cap R(WDD^*)\,\,\text{if and only if}\,\,\big(I+(\la-\rho) G\big)(\mc{R})=\mc{R}\cap R\big(I+(\la-\rho) G\big).
\end{equation}
\end{lema}

\begin{proof}
To prove the equivalence, note that
\[
R(A+\la B)=R(A+\rho B)\cap R(WDD^*)\quad\quad\text{if and only if}\quad\quad R(WDD^*W)=R(W^2)\cap R(WDD^*).
\]
The last equality holds if and only if $W\big(DD^*(\mc{R})\big)=W\big(\mc{R}\cap R(DD^*)\big)$,
or equivalently, $\big(I+(\la-\rho)G\big)(\mc{R})=\mc{R}\cap R\big(I+(\la-\rho)G\big)$,
where we used that $W$ is injective.
\end{proof}

\begin{obs} The first equality in \eqref{eq_dos_condiciones} can be expressed in terms of the operators $A$ and $B$.
In fact, by \eqref{eq:Douglas} we have that $WDD^*W=A+\la B$. Then $DD^*W=W^{-1}(A+\la B)$,
or equivalently,
\[
WDD^*=\big(W^{-1}(A+\la B)\big)^*.
\]
\end{obs}

\section{Quadratically constrained quadratic programming with a constraint}\label{section_problem}

Given two indefinite selfadjoint operators $A,B\in\mc{L}(\HH)$ and vectors $a,b\in\HH$,
consider the real valued functions
\[
f(x):=\PI{Ax}{x}+2\real\PI{a}{x}\quad\quad\text{and}\quad\quad g(x):=\PI{Bx}{x}+2\real\PI{b}{x},\qquad x\in\HH.
\]

The main purpose of this paper is to study the following QP1QC:

\begin{customprob}{1}
Given a constant $\beta\in\RR$,
analyze the existence of
\[
\min\,f(x)\qquad\text{subject to}\qquad g(x)\leq\beta.
\]
Denote by $\mc{W}(\beta,a,b)$ its set of solutions.
\end{customprob}

Under certain conditions the analysis of this problem reduces to analyzing the associated
problem with an equality constraint.

\begin{customprob}{2}
Given a constant $\beta\in\RR$, analyze the existence of
\[
\min\,f(x)\qquad\text{subject to}\qquad g(x)=\beta.
\]
Denote by $\ZZ(\beta,a,b)$ its set of solutions.
\end{customprob}

\smallskip

Using Fr\'echet derivatives we now show that if Problem \ref{pb 1} admits a solution, then the sets of solutions to
Problems \ref{pb 1} and \ref{pb 2} coincide.
The functions $f$ and $g$ are Fr\'echet differentiable at every $x\in\HH$ and the first and second order Fr\'echet derivatives of $f$
are given by
\[
Df(x)\ \dx =2\real(\PI{Ax+a}{\dx}), \qquad \dx\in\HH,
\]
\[
D^2f(x)\ (\dx_1,\dx_2) =2\real(\PI{A\Delta x_1}{\Delta x_2}), \qquad \Delta x_1,\Delta x_2\in\HH.
\]
The derivatives of $g$ can be expressed analogously.

\begin{prop}\label{prop_primer_equivalencia}
Given $\beta\in\RR$ and $(a,b)\in\HH\x\HH$, if $\mc{W}(\beta,a,b)\neq\varnothing$ then
\[
\min_{g(x)\leq\beta}f(x)=\min_{g(x)=\beta}f(x)\qquad\text{and}\qquad\mc{W}(\beta,a,b)=\ZZ(\beta,a,b).
\]
\end{prop}

\begin{proof}
Assume there exists $x_0\in \mc{W}(\beta,a,b)\setminus\ZZ(\beta,a,b)$, that is $g(x_0)<\beta$.
Then $f$ has a local minimum at $x_0$. By \cite[Prop. 2.4.18 and 2.4.19]{Marsden} we have that $Df(x_0)=0$ and
\[
0\leq D^2f(x_0)(\dx,\dx)=2\PI{A\dx}{\dx}\qquad\text{for every $\dx\in\HH$},
\]
but this is a contradiction because $A$ is indefinite. Then $x_0\in\ZZ(\beta,a,b)$, $\min_{g(x)\leq\beta}f(x)=\min_{g(x)=\beta}f(x)$ and
$\mc{W}(\beta,a,b)=\ZZ(\beta,a,b)$.
\end{proof}

The next theorem provides necessary and sufficient conditions for Problem \ref{pb 2}
to admit a solution,  as a direct consequence of Lagrange multipliers techniques.
We first show a technical result.

\begin{lema}\label{lema_cuenta_clave}
Given $(a,b)\in\HH\x\HH$, let $f$ and $g$ be as above. 
\begin{enumerate}[label=\roman*)]
\item If there exist $x\in\HH$ and $\la\in\RR$ such that $(A+ \la B)x=-(a+\la b)$, then
\[
(f+\la g)(y) - (f+\la g)(x)=\PI{(A+\la B)(y-x)}{y-x}\quad\quad\text{for every $y\in\HH$}.
\]
\item If there exists $x\in\HH$ such that $Bx=-b$, then 
\[
g(y)-g(x)=\PI{B(y-x)}{y-x}\quad\quad\text{for every $y\in\HH$}.
\]
\end{enumerate}
\end{lema}

\begin{proof}
\noi{\it i)} For $y\in\HH$,
\begin{align*}
&\big(f(y)+\lambda g(y)\big)-\big(f( x )+\lambda g( x )\big)=\nonumber\\
&=\big(\PI{(A+\lambda B)y}{y}+2\real\PI{a+\lambda b}{y}\big)-\big(\PI{(A+\lambda B) x }{ x }+2\real\PI{a+\lambda b}{ x }\big)\nonumber\\
&=\big(\PI{(A+\lambda B)y}{y}-2\real\PI{(A+\lambda B) x }{y}\big)-\big(\PI{(A+\lambda B) x }{ x }-2\PI{(A+\lambda B) x }{ x }\big)\nonumber\\
&=\PI{(A+\lambda B)y}{y}-2\real\PI{(A+\lambda B) x }{y}+\PI{(A+\lambda B) x }{ x }\nonumber\\
&=\PI{(A+\lambda B)(y- x )}{y- x }.
\end{align*}

\noi{\it ii)} It follows analogously.
\end{proof}

\begin{thm}\label{teo_derivada} 
Let $\beta\in\RR$ and $(a,b)\in\HH\x\HH$. Then, $ x \in \ZZ(\beta,a,b)$ if and only if there exists
$\lambda\in I_\geq(A,B)$ satisfying
\begin{equation}\label{eq:eq_normal}
\begin{cases}
\,(A+\lambda B) x =-(a+\lambda b),\\
\, \PI{B x }{ x }+2\real\PI{b}{ x }=\beta.
\end{cases}
\end{equation}
\end{thm}

\begin{proof}
Assume that there exist $\la\in I_\geq(A,B)$ and $x\in\HH$ satisfying $g(x)=\beta$
and $(A+\la B) x =-(a+\la b)$. If $y\in\HH$ is such that $g(y)=\beta$, then using Lemma \ref{lema_cuenta_clave},
\[
f(y)-f( x )=\big(f(y)+\lambda g(y)\big)-\big(f( x )+\lambda g( x )\big)=\PI{(A+\lambda B)(y- x )}{y- x }\geq0.
\]
Hence, $x\in\ZZ(\beta,a,b)$.

\smallskip

Conversely, suppose that $x\in\ZZ(\beta,a,b)$. If $Bx\neq -b$, then the proof follows the lines of the proof of \cite[Thm. 3.5]{GZMMP}.
Now assume that $B x =-b$, and let $y\in Q(B)$. Considering that $g( x )=\beta$ and $\PI{By}{y}=0$,
\begin{align}\label{eq:algo_0}
g( x +y)&=\PI{B( x +y)}{ x +y}+2\real\PI{b}{ x +y}\nonumber\\
&=\parentesis{\PI{B x }{ x }+2\real\PI{b}{ x }}+\PI{By}{y}+2\real\PI{B x }{y}+2\real\PI{b}{y}\nonumber\\
&=\beta+2\real\PI{B x }{y}+2\real\PI{b}{y}=\beta.
\end{align}
Let $\theta\in[0,2\pi)$ and $t\in\RR$, since $te^{i\theta}y\in Q(B)$,
\[
\PI{A x }{ x }+2\real\PI{a}{ x }=f( x )\leq f( x +te^{i\theta}y)=\PI{A( x +te^{i\theta}y)}{ x +te^{i\theta}y}+2\real\PI{a}{ x +te^{i\theta}y},
\]
which yields
\[
0\leq t^2\PI{Ay}{y}+2\,t\real\PI{A x +a}{e^{i\theta}y}.
\]
If $\theta\in[0,2\pi)$ is such that $\PI{A x +a}{e^{i\theta}y}=|\PI{A x +a}{y}|$, then
\[
0\leq t^2\PI{Ay}{y}+2\,t|\PI{A x +a}{y}|\qquad\text{for every $t\in\RR$.}
\]
It follows that $\PI{Ay}{y}\geq0$ and $\PI{A x +a}{y}=0$. This implies that $\PI{Ay}{y}\geq0$ for every
$y\in Q(B)$. Since $B$ is indefinite it then holds that $I_\geq(A,B)\neq\varnothing$ and $A x +a\in Q(B)^\bot=\set{0}$, i.e. $A x =-a$.
Then, choosing any $\la\in I_\geq(A,B)$ we have that $(A+\la B)x=-(a+\la b)$.
\end{proof}

From the above proof it is clear that the case in which there exists $x\in\HH$ such that
\begin{equation}\label{compa}
\left\{\begin{array}{rcl}
Ax &=& -a, \\
Bx &=& -b,
\end{array}\right.
\end{equation}

has to be treated separately.
Let $\mc{A}$ be the set of those $(a,b)\in\HH\x\HH$ such that there exists a solution to the system \eqref{compa}
and $\mc{A}^c=\HH\x\HH\setminus\mc{A}$. If Problem \ref{pb 2} admits a solution for $(a,b)\in\mc{A}$, then any $\la\in I_\geq(A,B)$
is a suitable Lagrange multiplier. But if $(a,b)\in\mc{A}^c$, then the Lagrange multiplier $\la$ in \eqref{eq:eq_normal}
is unique, as states the next proposition. Its proof, which is similar to the proof of \cite[Prop. 5.1]{GZMMP}, is included in the Appendix.

\begin{prop}\label{prop_unicidad_lambda} 
Given $\beta\in\RR$, let $(a,b)\in \mc{A}^c$ be such that $\ZZ(\beta,a,b)\neq\varnothing$.
Then, there exists a unique $\lambda\in I_\geq(A,B)$ such that
$$(A+\lambda B) x =-(a+\lambda b),$$
for every $ x \in \ZZ(\beta,a,b)$.
\end{prop}
\smallskip

In Proposition \ref{prop_primer_equivalencia} we show that the existence of solutions to Problem \ref{pb 1} guarantees the existence of solutions to Problem \ref{pb 2}.
We now state under which conditions Problem \ref{pb 2} admiting a solution implies that Problem \ref{pb 1} admits a solution.

\begin{thm}\label{teo_pb_equivalentes}
Given $\beta\in\RR$ and $(a,b)\in\HH\x\HH$, the following conditions are equivalent:
\begin{enumerate}[label=\roman*)]
\item $\mc{W}(\beta,a,b)\neq\varnothing$;
\item $\ZZ(\beta,a,b)\neq\varnothing$ and $\la_->0$.
\end{enumerate}
In this case, $\mc{W}(\beta,a,b)=\ZZ(\beta,a,b)$.
\end{thm}

\begin{proof}
Suppose that $\mc{W}(\beta,a,b)\neq\varnothing$.
By Proposition \ref{prop_primer_equivalencia}, $\ZZ(\beta,a,b)\neq\varnothing$, which implies that $I_\geq(A,B)\neq\varnothing$
and $\PI{Ax}{x}\geq0$ whenever $\PI{Bx}{x}=0$. Now 
let $y\in\mc{W}(\beta,a,b)$ and suppose that there exists $x\in\HH$ such that $\PI{Bx}{x}<0$ and $\PI{Ax}{x}<0$.
Then, there exists $\alpha_0>0$ such that
\[
g(\alpha x)=\alpha^2\PI{Bx}{x}+\alpha\,2\real\PI{b}{x}\leq\beta\qquad\text{for every $\alpha>\alpha_0$.}
\]
But
\[
f(y)\leq f(\alpha x)=\alpha^2\PI{Ax}{x}+\alpha\,2\real\PI{a}{x}\xrightarrow[\alpha\to+\infty]{}-\infty,
\]
and this is a contradiction.
Then $\PI{Az}{z}\geq0$ whenever $\PI{Bz}{z}\leq0$, and $\la_->0$ by Proposition \ref{prop_nueva_pencils}.

\smallskip

Conversely, assume that $\ZZ(\beta,a,b)\neq\varnothing$ and $I_\geq(A,B)=[\lambda_-,\lambda_+]$ with $0<\lambda_-\leq\lambda_+$.
If $\w{x}\in\ZZ(\beta,a,b)$, then there exists $\lambda\in[\lambda_-,\lambda_+]$ such that $(A+\lambda B)\w{x}=-(a+\lambda b)$.
Let $x\in\HH$ be such that $g(x)\leq\beta$. Considering that $g(\w{x})=\beta$ and $0<\lambda$, by Lemma \ref{lema_cuenta_clave},
\[
f(x)-f(\widetilde{x})\geq\big(f(x)+\lambda g(x)\big)-\big(f(\w{x})+\lambda g(\w{x})\big)=\PI{(A+\lambda B)(x-\w{x})}{x-\w{x}}\geq0.
\]
Hence, $\w{x}\in\mc{W}(\beta,a,b)$. Finally, by Proposition \ref{prop_primer_equivalencia} this implies that $\mc{W}(\beta,a,b)=\ZZ(\beta,a,b)$.
\end{proof}

Even in a finite dimensional setting, it is easy to construct examples where for given $\beta$ and $(a,b)$, $\mc{W}(\beta,a,b)$ is empty while $\ZZ(\beta,a,b)\neq\varnothing$.
									   
\smallskip

To finish this section, we describe the set of solutions $\ZZ(\beta,a,b)$ when $(a,b)\in\mc{A}$.
In the next section we show that under certain conditions the same behaviour is present when considering a larger set of initial data points.

\begin{prop}\label{prop_caso_trivial} Assume that $I_\geq(A,B)\neq\varnothing$. Given $\beta\in\RR$ and $(a,b)\in\mc{A}$, let $x_0\in\HH$ be such that $Ax_0=-a$ and $Bx_0=-b$. The following conditions
hold:
\begin{enumerate}[label=\roman*)]
\item if $g(x_0)=\beta$, then
\[
\ZZ(\beta,a,b)=x_0+Q(A)\cap Q(B);
\]
\item if $g(x_0)>\beta$, then $\ZZ(\beta,a,b)\neq\varnothing$ if and only if $N(A+\la_+ B)\neq\set{0}$. In this case,
\[
\ZZ(\beta,a,b)=x_0+\set{x\in N(A+\la_+ B)\,:\,\PI{Bx}{x}=\beta+\PI{Bx_0}{x_0}};
\]
\item if $g(x_0)<\beta$, then $\ZZ(\beta,a,b)\neq\varnothing$ if and only if $N(A+\la_- B)\neq\set{0}$. In this case,
\[
\ZZ(\beta,a,b)=x_0+\set{x\in N(A+\la_-B)\,:\,\PI{Bx}{x}=\beta+\PI{Bx_0}{x_0}}.
\]
\end{enumerate}
\end{prop}

\begin{proof} Denote $\nu=\beta+\PI{Bx_0}{x_0}$. 
Using that $Ax_0=-a$ and $Bx_0=-b$ yields
\[
f(x)=\PI{A(x-x_0)}{x-x_0}-\PI{Ax_0}{x_0}\qquad\text{and}\qquad g(x)=\PI{B(x-x_0)}{x-x_0}-\PI{Bx_0}{x_0}.
\]
Also, note that $g(x_0)=-\PI{Bx_0}{x_0}=\beta-\nu$. Then, $x\in\ZZ(\beta,a,b)$ if and only if $y:=x-x_0$ is a solution to
\[
\min \PI{Az}{z}\qquad\text{subject to}\qquad\PI{Bz}{z}=\nu.
\]
It follows that $\ZZ(\beta,a,b)\neq\varnothing$ if and only if $\ZZ(\nu,0,0)\neq\varnothing$, and in this case
\[
\mc{Z}(\beta,a,b)=x_0+\mc{Z}(\nu,0,0).
\]

\smallskip

\noi {\it i)} If $g(x_0)=\beta$, then $\nu=0$. Considering that $\PI{Ax}{x}\geq0$ whenever $\PI{Bx}{x}=0$, it follows that $\ZZ(\nu,0,0)=\ZZ(0,0,0)=Q(A)\cap Q(B)$.

\smallskip

\noi {\it ii)} The fact that $g(x)>\beta$ implies that $\nu<0$. Given $x\in\HH$, $x\in\ZZ(\nu,0,0)$ if and only if $(A+\la B)x=0$,
for some $\la\in [\la_-,\la_+]$, and $\PI{Bx}{x}=\nu<0$, see Theorem \ref{teo_derivada}. This only happens if $x\in N(A+\la B)\cap \mc{P}^-(B)$, or equivalently, $x\in N(A+\la_+B)$. Hence, $x\in\ZZ(\nu,0,0)$
if and only if $x\in N(A+\la_+B)$ is such that $\PI{Bx}{x}=\nu$.

\smallskip

\noi {\it iii)} The result follows using similar arguments as above.
\end{proof}

\section{Necessary and sufficient conditions for the existence of solutions}\label{section_conditions}

The aim of this section is to provide necessary and sufficient conditions for Problem \ref{pb 2} to admit a solution for
an appropiate set of initial data points $(a,b)$. In the first place,
if $b=Bx_0$ for some $x_0$ then $g(x)=\PI{B(x-x_0)}{x-x_0}+g(x_0)$, and $x\in\ZZ(\beta,a,b)$ if and only if $x+x_0\in\ZZ(\nu,a,0)$ for
$\nu=\PI{Bx_0}{x_0}+\beta$.
Hence, we focus on vectors $b$ belonging to $R(B)$.
In the second place, if $\ZZ(\nu,a,0)\neq\varnothing$, then by Theorem \ref{teo_derivada} there exists $\la\in I_\geq(A,B)$ such that
$a\in R(A+\la B)$. Since, for $I_\geq(A,B)=[\la_-,\la_+]$ with $\la_-<\la_+$,
\[
\text{span}\bigg\{\bigcup_{\la\in [\la_-,\la_+]}R(A+\la B)\bigg\}=R(A)+R(B),
\]
we consider vectors $a$ belonging to $R(A)+R(B)$.

Given $\beta\in\RR$ and vectors $a\in R(A)+R(B)$, $b\in R(B)$, if Problem \ref{pb 2} admits a solution then the set of solutions contains an affine manifold parallel to
$N(A)\cap N(B)$. Indeed, if $x_{(a,b)}\in\ZZ(\beta,a,b)$ then
\[
x_{(a,b)}+N(A)\cap N(B)\subseteq \ZZ(\beta,a,b).
\]
From now on we assume that $N(A)\cap N(B)=\set{0}$, though in Subsection \ref{sub_remarks} we discuss how the results can be easily expressed for the general case.

\medskip

We start by giving alternative formulations for the problem under consideration in this section.

\begin{prop}\label{prop_pivote} The following conditions are equivalent:
\begin{enumerate}[label=\roman*)]
\item for a fixed $\beta\in\RR$, $\ZZ(\beta,a,b)\neq\varnothing$ for every $a\in R(A)+R(B)$ and every $b\in R(B)$;
\item for each $\nu\in\RR$, $\ZZ(\nu,a,0)\neq\varnothing$ for every $a\in R(A)+R(B)$;
\item for each $\beta'\in\RR$, $\ZZ(\beta',a,b)\neq\varnothing$ for every $a\in R(A)+R(B)$ and every $b\in R(B)$.
\end{enumerate}
\end{prop}

\begin{proof}
\noi {\it i)$\to$ii)} Given $\nu\in\RR$ and $a\in R(A)+R(B)$, let $x_1,x_2,x_3\in\HH$ be such that
\[
a=Ax_1+Bx_2\qquad\text{and}\qquad\PI{Bx_3}{x_3}=\nu-\beta.
\] 
Since $\ZZ\big(\beta,A(x_1+x_3)+Bx_2,Bx_3\big)\neq\varnothing$, by Theorem \ref{teo_derivada},
there exist $\la\in I_\geq(A,B)$ and $x\in\HH$ satisfying
\[
\begin{cases}
\, (A+\la B)x = -\big(A(x_1+x_3)+Bx_2+\la Bx_3\big),\\
\, \PI{Bx}{x}+2\real\PI{Bx_3}{x}=\beta,
\end{cases}
\]
or equivalently, using that $\nu=\PI{Bx_3}{x_3}+\beta$,
\begin{equation}\label{eq:sistema_para_obs}
\begin{cases}
\,(A+\la B)(x+x_3) = -(Ax_1+Bx_2)=-a,\\
\,\PI{B(x+x_3)}{x+x_3}=\nu.
\end{cases}
\end{equation}
Hence, $x+x_3\in\ZZ(\nu,a,0)$.

\smallskip

\noi {\it ii)$\to$iii)} Given $\beta'\in\RR$, $a\in R(A)+R(B)$ and $b\in R(B)$, let $x_1,x_2,x_3\in\HH$ and $\nu\in\RR$ be such that
\[
a=Ax_1+Bx_2,\qquad b=Bx_3 \qquad\text{and}\qquad\PI{Bx_3}{x_3}=\nu-\beta'.
\]
Since $\ZZ\big(\nu,A(x_1-x_3)+Bx_2,0\big)\neq\varnothing$, by Theorem \ref{teo_derivada}, there exist $\la\in I_{\geq}(A,B)$ and $x\in\HH$ satisfying
\[
\begin{cases}
\,(A+\la B)x = -\big(A(x_1-x_3)+Bx_2\big),\\
\,\PI{Bx}{x}=\nu.
\end{cases}
\]
Using that $\nu=\beta'+\PI{Bx_3}{x_3}$,
\[
\begin{cases}
\,(A+\la B)(x-x_3) = -\big(Ax_1+Bx_2+\la Bx_3\big)=-(a+\la b),\\
\,\PI{B(x-x_3)}{x-x_3}+2\real\PI{Bx_3}{x-x_3}=\beta',
\end{cases}
\]
which implies $x-x_3\in\ZZ(\beta',a,b)$.

\smallskip

\noi {\it iii)$\to$i)} This is trivial.

\end{proof}

\begin{obs}\label{obs_pivote} From the proof of Proposition \ref{prop_pivote} follows a fact that we often use in the rest of the section.
Namely, given $\beta\in\RR$, $a\in R(A)+R(B)$ and $b\in R(B)$, if 
$x_1,x_2,x_3\in\HH$ are such that $a=A(x_1+x_3)+Bx_2$ and $b=Bx_3$, then $\ZZ(\beta,a,b)\neq\varnothing$
if and only if there exists $(\la,y)\in I_\geq(A,B)\x\HH$ satisfying
\begin{equation}\label{eq_sistema_con_c}
\begin{cases}
\,(A+\la B)y = c,\\
\,\PI{By}{y}=\nu,
\end{cases}
\end{equation}
with
\begin{equation}\label{eq:u0}
c=-(Ax_1+Bx_2)\quad\quad\text{and}\quad\quad \nu=\PI{Bx_3}{x_3}+\beta
\end{equation}
To see this, take $y=x+x_3$ in \eqref{eq:sistema_para_obs}.
\end{obs}

The next lemma establishes some necessary conditions for Problem \ref{pb 2} admitting a solution for any of the data points
of interest.

\begin{lema}\label{lema_condiciones_necesarias}
Given $\beta\in\RR$, if $\ZZ(\beta,a,b)\neq\varnothing$ for every $a\in R(A)+R(B)$ and every $b\in R(B)$, then the following conditions hold:
\begin{enumerate}[label=\roman*)]
\item $I_\geq(A,B)=[\la_-,\la_+]$ with $\lambda_-<\lambda_+$;
\item $N(A+\la_-B)\neq\set{0}$ and $N(A+\la_+B)\neq\set{0}$.
\end{enumerate}
\end{lema}

\begin{proof}
By Theorem \ref{teo_derivada}, $I_\geq(A,B)\neq\varnothing$. If $\la_-=\la_+=\la$, then $N(A+\la B)=\set{0}$. In fact,
given $a\in R(A)+R(B)$ there exists $x\in\HH$ satisfying $g(x)=\beta$ and
\[
(A+\lambda B)x=-a,
\]
because $\ZZ(\beta,a,0)\neq\varnothing$.
Since $a$ is arbitrary, $R(A+\lambda B)=R(A)+R(B)$, and thus $N(A+\la B)=\set{0}$.

Now consider $x_0\in\HH$ such that $\PI{Bx_0}{x_0}>-\beta$. Since $\ZZ(\beta,a,b)\neq\varnothing$ for every $a\in R(A)+R(B)$ and every $b\in R(B)$, by item {\it ii)}
of Proposition \ref{prop_pivote} we have that $\ZZ(\nu,0,0)\neq\varnothing$, with $\nu=\PI{Bx_0}{x_0}+\beta$. Then, there exist
$\la\in I_\geq(A,B)=[\la_-,\la_+]$ and $y \in\HH$ satisfying
\[
(A+\la B)y=0\qquad\text{and}\qquad \PI{By}{y}=\nu>0.
\]
Hence, $y\in N(A+\la B)\cap \mc{P}^+(B)$. By Proposition \ref{nucleo del interior},
this can only happen if $\la=\la_-\neq\la_+$. Then $y\in N(A+\la_-B)$.

Taking a vector $x_0\in\HH$ such that $\PI{Bx_0}{x_0}<-\beta$ and following the same procedure we get that $N(A+\la_+B)\neq\set{0}$.
\end{proof}

\begin{hyp2}\label{hip2_0} From now on, we assume that
\[
I_\geq(A,B)=[\la_-,\la_+]\qquad\text{with $\la_-<\la_+$.}
\]
\end{hyp2}

Fixing $\rho:=\tfrac{\la_-+\la_+}{2}$, consider the operator $W$, the subspace $\mc{R}$ defined in \eqref{eq:def_W}, and 
the selfadjoint operator $G\in\mc{L}(\HH)$ given in Lemma \ref{lema_G_acotado}.
If $\kappa:=\tfrac{\la_+-\la_-}{2}$, then $I_\geq(I,G)=[\la_--\rho,\la_+-\rho]=[-\kappa,\kappa]$.
Also, by Corollary \ref{cor:lambdas_distintos_0}, it follows that
\[
R(A)+R(B)\subseteq\mc{R}.
\]
The next lemma shows that \eqref{eq_sistema_con_c} can be simplified by expressing it in an equivalent way.

\begin{lema}\label{lema_nueva_normal} Given $\nu\in\RR$ and $c\in R(A)+R(B)$, define the vector $u_0:=W^{-1}c$.
Then, there exists $(\la,x)\in[\lambda_-,\lambda_+]\x\HH$ satisfying
\[
\begin{cases}
\,(A+\la B)x = c,\\
\,\PI{Bx}{x}=\nu,
\end{cases}
\]
if and only if there exists $(\gamma,y)\in [-\kappa,\kappa]\x\mc{R}$ satisfying
\begin{equation}\label{eq:nuevo_sistema}
\begin{cases}
\,(I+\gamma G)y=u_0,\\
\,\PI{Gy}{y}=\nu.
\end{cases}
\end{equation}
Therefore, given $\beta\in\RR$, $a\in R(A)+R(B)$ and $b\in R(B)$,
$\ZZ(\beta,a,b)\neq\varnothing$ if and only if there exists $(\gamma,y)\in[-\kappa,\kappa]\x\mc{R}$ satisfying \eqref{eq:nuevo_sistema} with
$u_0=W^{-1}c$, where $c$ and $\nu$ are given by \eqref{eq:u0}.
\end{lema}

\begin{proof}
Let $x\in\HH$ and $\lambda\in[\lambda_-,\lambda_+]$. The equation $(A+\lambda B)x=c$
is satisfied if and only if
\[
W\big(I+(\lambda-\rho)G\big)Wx=c,
\]
see \eqref{eq:pencil_alternativo}.
If $y:=Wx$ and $\gamma:=\lambda-\rho\in[-\kappa,\kappa]$, this last equation is equivalent to
\[
(I+\gamma G)y=u_0.
\]
Also, $\PI{Gy}{y}=\PI{Bx}{x}=\PI{WGWx}{x}=\PI{Bx}{x}=\nu$, completing the proof.

Finally, fix $\beta\in\RR$, take $a\in R(A)+R(B)$ and $b\in R(B)$, and let $x_1,x_2,x_3\in\HH$ be such that $b=Bx_3$ and $a=A(x_1+x_3)+Bx_2$.
Defining $u_0:=-W^{-1}(Ax_1+Bx_2)$ and $\nu:=\PI{Bx_3}{x_3}+\beta$, using Remark \ref{obs_pivote} we get that
$\ZZ(\beta,a,b)\neq\varnothing$ if and only if there exists $(\gamma,y)\in[-\kappa,\kappa]\x\mc{R}$ satisfying \eqref{eq:nuevo_sistema}.
\end{proof}




\begin{prop}\label{prop_unicidad_2} 
Given $\nu\in\RR$ and $u_0\in\HH$, assume that there exists $(\gamma_1, y_1)\in [-\kappa,\kappa ]\x \HH$ satisfying \eqref{eq:nuevo_sistema}.
If $u_0\notin N(G)$ and $(\gamma_2,y_2)\in[-\kappa,\kappa]\x\HH$ also satisfies \eqref{eq:nuevo_sistema}, then $\gamma_1=\gamma_2$.
\end{prop}

\begin{proof}
Let $u_0\notin N(G)$, and assume there exist $(\gamma_1,y_1),(\gamma_2,y_2)\in[-\kappa,\kappa]\x\HH$ satisfying
$\PI{Gy_1}{y_1}=\PI{Gy_2}{y_2}=\nu$ and
\begin{align}\label{eq:sistemax}
(I+\gamma_1 G)y_1&=u_0,\nonumber\\
(I+\gamma_2 G)y_2&=u_0.
\end{align}
On the one hand, since $\PI{Gy_i}{y_i}=\nu$ and $I+\gamma_iG$ is positive semidefinite for $i=1,2$, $\PI{u_0}{y_i}=\|y_i\|^2+\gamma_i\,\nu\geq0$. On the other hand,
\begin{align}\label{eq:sistemax_2}
\|y_1\|^2+\gamma_1\,\nu=\PI{u_0}{y_1}=\PI{(I+\gamma_2 G)y_2}{y_1}=\PI{y_2}{y_1}+\gamma_2\PI{Gy_2}{y_1},\nonumber\\
\|y_2\|^2+\gamma_2\,\nu=\PI{u_0}{y_2}=\PI{(I+\gamma_1 G)y_1}{y_2}=\PI{y_1}{y_2}+\gamma_1\PI{Gy_1}{y_2},
\end{align}
This implies that
\begin{equation}\label{eq:gammas}
(\gamma_1-\gamma_2)\big(\PI{Gy_1}{y_2}+\nu\big)=\|y_2\|^2-\|y_1\|^2.
\end{equation}
Next, we see that $\|y_1\|=\|y_2\|$. Consider $a_{ij}:=\|(I+\gamma_j G)^{1/2}y_i\|$ for $i,j=1,2$. By the Cauchy-Schwarz inequality, and the fact that
$\|(I+\gamma_j G)^{1/2}y_i\|^2=\|y_i\|^2+\gamma_j\nu$
it follows that
\begin{align*}
a_{11}^2&=|\PI{u_0}{y_1}|=|\PI{(I+\gamma_2 G)y_2}{y_1}|\leq(\|y_2\|^2+\gamma_2\,\nu)^{1/2}(\|y_1\|^2+\gamma_2\,\nu)^{1/2}=a_{22}a_{12},\\
a_{22}^2&=|\PI{u_0}{y_2}|=|\PI{(I+\gamma_1 G)y_1}{y_2}|\leq(\|y_1\|^2+\gamma_1\,\nu)^{1/2}(\|y_2\|^2+\gamma_1\,\nu)^{1/2}=a_{11}a_{21}.
\end{align*}
By adding these inequalities, we get $a_{11}^2+a_{22}^2\leq a_{22}a_{12}+a_{11}a_{21}$, or equivalently,
\[
(a_{11}-a_{21})^2+(a_{22}-a_{12})^2\leq a_{21}^2+a_{12}^2-a_{11}^2-a_{22}^2=0.
\]
Then $a_{11}=a_{21}$ and $a_{22}=a_{12}$, i.e. $\|y_1\|=\|y_2\|$. By \eqref{eq:gammas}, this implies that $\gamma_1=\gamma_2$ or $\PI{Gy_1}{y_2}=-\nu$.
However, if $\PI{Gy_1}{y_2}=-\nu$ then \eqref{eq:sistemax_2} says that
\[
\PI{y_1}{y_2}=\|y_1\|^2+(\gamma_1+\gamma_2)\nu,
\]
and consequently, if $C=I+\tfrac{\gamma_1+\gamma_2}{2}G$ then
\begin{align*}
\PI{C^{1/2}y_1}{C^{1/2}y_2}&=\PI{Cy_1}{y_2}=\PI{y_1}{y_2}-\tfrac{\gamma_1+\gamma_2}{2}\nu=\|y_1\|^2+\tfrac{\gamma_1+\gamma_2}{2}\nu\\
&=(\|y_1\|^2+\tfrac{\gamma_1+\gamma_2}{2}\nu)^{1/2}(\|y_2\|^2+\tfrac{\gamma_1+\gamma_2}{2}\nu)^{1/2}\\
&=\|C^{1/2}y_1\|\,\|C^{1/2}y_2\|.
\end{align*}
Since $C$ is positive definite, then $y_1$ and $y_2$ are collinear and from \eqref{eq:sistemax_2} it is easy to see that
$y_1=y_2$, which by \eqref{eq:sistemax} implies that
$(\gamma_1-\gamma_2)Gy_1=0$. But $y_1\notin N(G)$ because $u_0\notin N(G)$, and hence $\gamma_1=\gamma_2$.
\end{proof}

The existence of solutions to Problem \ref{pb 2} for every $a\in R(A)+R(B)$ and every $b\in R(B)$ implies a range additivity condition for the
ranges of the pencil $A+\la B$.

\begin{prop}\label{prop_rangos_iguales}
Given $\beta\in\RR$, assume that $\ZZ(\beta,a,b)\neq\varnothing$ for every $a\in R(A)+R(B)$ and every $b\in R(B)$. Then,
\begin{equation}\label{eq:rangos_iguales_2}
R(A+\la B)=R(A)+R(B)\quad\quad\text{for every $\la\in(\la_-,\la_+)$.}
\end{equation}
\end{prop}

\begin{proof}
By Proposition \ref{prop:rango_invariante}, it is sufficient to prove that
\[
G(\mc{R})\subseteq (I+\tau G)(\mc{R})\quad\quad\text{for every $\tau\in(-\kappa,\kappa)$.}
\]
Fix $x_0\in\mc{R}$. If $x_0\in N(G)$, then $Gx_0=0\in(I+\tau G)(\mc{R})$ for every $\tau\in(-\kappa,\kappa)$. Now suppose that $x_0\notin N(G)$ and
fix $\tau_0\in(-\kappa,\kappa)$. Write $x_0=x_0^++x_0^-+x_0^0$ with $x_0^\pm\in\HH_\pm$ and $x_0^0\in N(G)$.
Consider the canonical decomposition of $G$ with respect to \eqref{eq:descomposicion}, and the real constant $\nu$ defined by
\[
\nu:=\|G_+^{1/2}(I_++\tau_0 G_+)^{-1}G_+ x_0^+\|^2-\|G_-^{1/2}(I_--\tau_0 G_-)^{-1}G_- x_0^-\|^2,
\]
where the operators $I_++\tau_0 G_+$ and $I_--\tau_0 G_-$ are invertible because $I+\tau_0 G$ is invertible by \eqref{eq:pencil_alternativo}.
On the one hand, setting
\[
z_0:=(I_++\tau_0 G_+)^{-1}G_+x_0^++(I_--\tau_0 G_-)^{-1}G_-x_0^-
\]
yields
\[
\begin{cases}
\,(I+\tau_0 G)z_0=Gx_0,\\
\,\PI{Gz_0}{z_0}=\nu.
\end{cases}
\]
On the other hand, by Lemma \ref{lema_nueva_normal}, there exists $(\gamma,y)\in [-\kappa,\kappa]\x\mc{R}$ satisfying
\begin{equation}\label{eq:un_sistema}
\begin{cases}
\,(I+\gamma G)y=Gx_0,\\
\,\PI{Gy}{y}=\nu.
\end{cases}
\end{equation}
Hence, both pairs $(\tau_0,z_0)$ and $(\gamma,y)$ are solutions to the system
\begin{equation}\label{eq:sistema_auxiliar}
\begin{cases}
\,(I+\tau G)z=Gx_0,\\
\,\PI{Gz}{z}=\nu.
\end{cases}
\end{equation}
Also, $Gx_0\notin N(G)$ because $N(G^2)=N(G)$ and $x_0\notin N(G)$.
By Proposition \ref{prop_unicidad_2},
this implies that $\gamma=\tau_0$, and thus
\[
Gx_0=(I+\tau_0G)y\in (I+\tau_0 G)(\mc{R}).
\]
Since $x_0\in\mc{R}$ and $\tau_0\in(-\kappa,\kappa)$ were arbitrary, the proof is completed.
\end{proof}

The next proposition states the necessary range additivity condition for $\la=\la_-$ and $\la=\la_+$.
Let $D_\pm\in\mc{L}(\HH)$ be the solution to $(A+\la_\pm B)^{1/2}=WX$.

\begin{prop}\label{prop_necesaria_extremos}
Given $\beta\in\RR$, assume that $\ZZ(\beta,a,b)\neq\varnothing$ for every $a\in R(A)+R(B)$ and every $b\in R(B)$. Then,
\[
R(A+\la_\pm B)=\big(R(A)+R(B)\big)\cap R(WD_\pm D_\pm^*).
\]
\end{prop}

\begin{proof}
We prove that $(I-\kappa G)(\mc{R})=\mc{R}\cap R(I-\kappa G)$, then the statement for $\la_-$ follows from Lemma
\ref{lema_condiciones_rangos}. A similar argument can be used to prove the other equality.

By Proposition \ref{prop_rangos_iguales}, $R(W^2)=R(A)+R(B)$, and by Proposition \ref{prop:rango_invariante}, $G(\mc{R})\subseteq\mc{R}$. Then, the inclusion
\[
(I- \kappa G)(\mc{R})\subseteq \mc{R}\cap R(I-\kappa G)
\]
is trivial. To see the other inclusion, take $x_0\in \mc{R}\cap R(I-\kappa G)$, and
write $x_0=x_0^++x_0^-+x_0^0$ with $x_0^\pm\in R(I_\pm\mp\kappa G_\pm)$ and $x_0^0\in N(G)$.
Consider the real valued functions $g_+$ and $g_-$ defined by
\[
g_+(\tau)=\|G_+^{1/2}(I_++\tau G_+)^{-1} x_0^+\|^2,\quad \tau\in(-\kappa,\kappa],
\]
and
\[
g_-(\tau)=\|G_-^{1/2}(I_--\tau G_-)^{-1} x_0^-\|^2,\quad \tau\in[-\kappa,\kappa),
\]
respectively. Since $x_0^+\in R(I_+-\kappa G_+)$, 
$\lim_{\tau\to-\kappa}g_+(\tau)=\|G_+^{1/2}(I_+-\kappa G_+)^{\dag} x_0^+\|^2$, and thus the extension
of $g_+$ is a continuous function on the interval $[-\kappa,\kappa]$. Now, by the functional calculus for selfadjoint operators,
$g_+$ is a continuous monotone decreasing function of $\tau$, and $g_-$ is a continuous monotone increasing
function of $\tau$. Hence, choosing $\nu\in\RR$ such that $g_+(-\kappa)<g_-(-\kappa)+\nu$ we have that
\begin{equation}\label{eq:desigualdad}
\max_{\tau\in[-\kappa,\kappa]}g_+(\tau)=g_+(-\kappa)<g_-(-\kappa)+\nu=\min_{\tau\in[-\kappa,\kappa)}g_-(\tau)+\nu.
\end{equation}
By Lemma \ref{lema_nueva_normal}, there exists $(\gamma,y)\in[-\kappa,\kappa]\x \mc{R}$ satisfying
\begin{equation*}
\begin{cases}
\,(I+\gamma G)y=x_0,\\
\,\PI{Gy}{y}=\nu.
\end{cases}
\end{equation*}
Suppose first that $\gamma\in(-\kappa,\kappa)$. Then,
\[
y=(I_++\gamma G_+)^{-1}x_0^++(I_--\gamma G_-)^{-1}x_0^-+x_0^0,
\]
and
\[
g_+(\gamma)-g_-(\gamma)=\|G_+^{1/2}(I_++\gamma G_+)^{-1} x_0^+\|^2-\|G_-^{1/2}(I_--\gamma G_.)^{-1} x_0^.\|^2=\PI{Gy}{y}=\nu.
\]
But by \eqref{eq:desigualdad} this is a contradiction. Now suppose that $\gamma=\kappa$. Then
$x_0^-\in R(I_--\kappa G_-)$, and
\[
y=(I_++\kappa G_+)^{-1}x_0^++(I_--\kappa G_-)^\dag x_0^-+x_0^0+y_-,
\]
for some $y_-\in N(I_--\kappa G_-)$.
Since $\lim_{\tau\to\kappa}g_-(\tau)=\|G_-^{1/2}(I_--\kappa G_-)^{\dag} x_0^-\|^2$, the extension
of $g_-$ is a continuous function on the interval $[-\kappa,\kappa]$. Considering that $G_-y_-=\tfrac{1}{\kappa}y_-$ and $(I_--\kappa G_-)^\dag x_0^-\bot y_-$,
we have that
\begin{align*}
\nu&=\PI{Gy}{y}=\|G_+^{1/2}(I_++\kappa G_+)^{-1}x_0^+\|^2-\|G_-^{1/2}\big((I_--\kappa G_-)^\dag x_0^-+y_-)\|^2\\
&=g_+(\kappa)-g_-(\kappa)-\tfrac{1}{\kappa}\|y_-\|^2\leq g_+(\kappa)-g_-(\kappa),
\end{align*}
and thus
\[
g_+(\kappa)\geq g_-(\kappa)+\nu.
\]
But by \eqref{eq:desigualdad} this leads to a contradiction. Hence, $\gamma=-\kappa$ and $x_0=(I-\kappa G)y\in (I-\kappa G)(\mc{R})$.
Since $x_0\in\mc{R}\cap R(I-\kappa G)$ was arbitrary, it follows that
\[
\mc{R}\cap R(I-\kappa G)\subseteq (I-\kappa G)(\mc{R}).\qedhere
\]
\end{proof}

Consider the subspaces $\N_+:=N(I-\kappa G)$ and $\N_-:=N(I+\kappa G)$. 
Then,
\begin{equation}\label{eq:nucleos_buenos}
W\big(N(A+\lambda_+ B)\big)=\N_-\cap\mc{R}\qquad\text{and}\qquad W\big(N(A+\la_-B)\big)=\N_+\cap\mc{R}.
\end{equation}

We now establish necessary and sufficient conditions for the system \eqref{eq:nuevo_sistema} to admit a solution for every $u_0\in\HH$ and every $\nu\in\RR$.
The proof of the next proposition uses the same arguments used in the proof of \cite[Prop. 4.8]{GZMMP}, and can be found in the Appendix.

\begin{prop}\label{prop_necesarias_suficientes_normal} For every $u_0\in\HH$ and every $\nu\in\RR$ there exists $(\gamma,y)\in[-\kappa,\kappa]\x\HH$
satisfying
\[
\begin{cases}
\,(I+\gamma G)y=u_0,\\
\,\PI{Gy}{y}=\nu,
\end{cases}
\]
if and only if $\mc{N}_-\neq\set{0}$ and $\mc{N}_+\neq\set{0}$.
\end{prop}

The following theorem shows that the necessary conditions established in Lemma \ref{lema_condiciones_necesarias} and Propositions
\ref{prop_rangos_iguales} and \ref{prop_necesaria_extremos} are also sufficient
for Problem \ref{pb 2} to admit a solution for every $a\in R(A)+R(B)$ and
every $b\in R(B)$, for a given $\beta\in\RR$. Note that if $\la\in(\la_-,\la_+)$ and $D\in\mc{L}(\HH)$ is the solution to $(A+\la B)^{1/2}=WX$, the condition
\[
R(A+\la B)=\big(R(A)+R(B)\big)\cap R(WDD^*)
\]
is equivalent to $R(A+\la B)=R(A)+R(B)$.

\begin{thm}\label{teo_necesarias_suficientes} Given $\beta\in\RR$, $\ZZ(\beta,a,b)\neq\varnothing$ for every $a\in R(A)+R(B)$ and every $b\in R(B)$
if and only if the following conditions hold:
\begin{enumerate}[label=\roman*)]
\item $I_\geq(A,B)=[\la_-,\la_+]$ with $\la_-<\la_+$;
\item $N(A+\la_-B)\neq\set{0}$ and $N(A+\la_+B)\neq\set{0}$;
\item for every $\la\in[\la_-,\la_+]$, $R(A+\la B)=\big(R(A)+R(B)\big)\cap R(WDD^*)$,
where $D\in\mc{L}(\HH)$ is the solution to $(A+\la B)^{1/2}=WX$.
\end{enumerate}
\end{thm}

\begin{proof}

The necessity of the conditions follows from
Lema \ref{lema_condiciones_necesarias}, and Propositions \ref{prop_rangos_iguales} and \ref{prop_necesaria_extremos}.

Assume that conditions {\it i)}, {\it ii)} and {\it iii)} hold. Fix $\beta\in\RR$, and let $a\in R(A)+R(B)$ and $b\in R(B)$. Consider the vector $c\in R(A)+R(B)$ and the constant $\nu\in\RR$ defined in \eqref{eq:u0},
and set $u_0=W^{-1}c$. Since $R(W^2)=R(A)+R(B)$, we have that $u_0\in W^{-1}\big(R(W^2)\big)=\mc{R}$.
By Lemma \ref{lema_nueva_normal}, it is sufficient to show that there exists $(\gamma,y)\in[-\kappa,\kappa]\x\mc{R}$ satisfying
\begin{equation}\label{eq:sistema_2}
\begin{cases}
\,(I+\gamma G)y=u_0,\\
\,\PI{Gy}{y}=\nu,
\end{cases}
\end{equation}
in order to prove that $y\in\ZZ(\beta,a,b)$. Since {\it ii)} implies that $\mc{N}_+\cap\mc{R}\neq\set{0}$ and $\mc{N}_-\cap\mc{R}\neq\set{0}$,
there exists $(\gamma,y)\in[-\kappa,\kappa]\x\HH$ such that $(\gamma,y)$ satisfies \eqref{eq:sistema_2}, see Proposition \ref{prop_necesarias_suficientes_normal}.
If $\gamma\in(-\kappa,\kappa)$, then $y$ is unique and
since $(I+\gamma G)y=u_0\in\mc{R}$, by Proposition \ref{prop:rango_invariante} we have that $y\in \mc{R}$, and thus $y\in \ZZ(\beta,a,b)$.

By {\it iii)} and Lemma \ref{lema_condiciones_rangos}, if $\gamma=\kappa$ then $u_0\in\mc{R}\cap R(I+\kappa G)=(I+\kappa G)(\mc{R})$.
Thus, there exists $z_0\in \N_-$ such that
\[
w_0:=(I+\kappa G)^\dag u_0+z_0\in\mc{R}.
\]
If $\PI{Gw_0}{w_0}=\nu$, then the pair $(\kappa,w_0)$ satisfies \eqref{eq:sistema_2}, and $w_0\in\ZZ(\beta,a,b)$.
Now assume that $\PI{Gw_0}{w_0}>\nu$, and choose $z\in\mc{N}_-\cap \mc{R}$ with $\|z\|=1$.
Setting
\[
\alpha=-\real\PI{w_0}{z}+\Big(\big(\real\PI{w_0}{z}\big)^2+\kappa\big(\PI{Gw_0}{w_0}-\nu\big)\Big)^{1/2},
\]
we get that the pair $(\kappa,w_0+\alpha z)$ satisfies \eqref{eq:sistema_2}, and also $w_0+\alpha z\in\mc{R}$. Thus, $w_0+\alpha z\in\ZZ(\beta,a,b)$.
If $\PI{Gw_0}{w_0}<\nu$, since $N(A+\la_-B)\neq\set{0}$ we can follow the same procedure for some $z\in\N_+\cap\mc{R}$ with $\|z\|=1$.

A similar argument holds if we assume that $\gamma=-\kappa$, thus proving the sufficiency.
\end{proof}

 In the proof above we showed that, under the stated conditions, given $u_0\in\mc{R}$ and $\nu\in\RR$
there exists a pair $(\gamma,y)\in[-\kappa,\kappa]\times\mc{R}$
that satisfies \eqref{eq:sistema_2}.
On the one hand, if $\gamma\in(-\kappa,\kappa)$ then the vector $y$ is unique. On the other hand, if
$\gamma=\kappa$ then there exist $\alpha\geq0$ and $\w{y}\in\mc{R}$ such that any vector
\[
y\in\w{y}+\alpha \cdot\N_-\cap\mc{R}\cap\St
\]
satisfies that $(\kappa,y)$ is a solution to \eqref{eq:sistema_2}, where $\St$ denotes the unit sphere $\St=\set{x\in\HH\,:\,\|x\|=1}$.
An analogous statement holds for $\gamma=-\kappa$.

\subsection{Final remarks}\label{sub_remarks}

\begin{enumerate}[wide, labelwidth=!, labelindent=0pt]
\item If $R(A+\la_0B)$ is closed for some $\la_0\in(\la_-,\la_+)$, then $R(A+\la B)$ is closed for every $\la\in(\la_-,\la_+)$, see \cite[Cor. 3.7]{GZMMP2}.
In this case,
\[
R(A+\la B)=R(A)+R(B)=\HH\qquad\text{for every $\la\in(\la_-,\la_+)$.}
\]
Also, $W$ is invertible and it is immediate that
\[
R(A+\la_-B)=R(WD_-D_-^*)=\big(R(A)+R(B)\big)\cap R(WD_-D_-^*)
\]
and
\[
R(A+\la_+B)=R(WD_+D_+^*)=\big(R(A)+R(B)\big)\cap R(WD_+D_+^*).
\]
Then, following arguments similar to the ones used in the proof of \cite[Prop. 4.3]{GZMMP} we get the next theorem, which shows that if we enlarge
the set of admissible data points to $\mc{R}\x R(B)$, then $R(A+\la B)=\HH$ for any $\la\in(\la_-,\la_+)$. 

\begin{thm}\label{teo_final}Given $\beta\in\RR$, $\ZZ(\beta,a,b)\neq\varnothing$ for every $a\in\mc{R}$ and every $b\in R(B)$
if and only if the following conditions hold:
\begin{enumerate}[label=\roman*)]
\item $I_\geq(A,B)=[\la_-,\la_+]$ with $\la_-<\la_+$;
\item $N(A+\la_-B)\neq\set{0}$ and $N(A+\la_+B)\neq\set{0}$;
\item $R(A+\la B)=\HH$ for some $\la\in(\la_-,\la_+)$.
\end{enumerate}
\end{thm}

\begin{proof} The sufficiency follows from the discussion above. To prove the converse, assume that $\ZZ(\beta,a,b)\neq\varnothing$ for every $a\in\mc{R}$ and
every $b\in R(B)$. By Theorem \ref{teo_necesarias_suficientes}, we only have to prove {\it iii)}. 
By Proposition \ref{prop_rangos_iguales}, $R(A+\la B)=R(A)+R(B)$ for every $\la\in(\la_-,\la_+)$.
Let $a\in\mc{R}$, since $\ZZ(\beta,a,0)\neq0$ by Theorem \ref{teo_derivada} there exists $(\la,x)\in[\la_-,\la_+]\x\HH$
satisfying $g(x)=\beta$ and
\[
-a=(A+\la B)x\in R(A)+R(B).
\]
Since $a\in\mc{R}$ is arbitrary, it follows that $R\big((A+\rho B)^{1/2}\big)=\mc{R}\subseteq R(A)+R(B)=R(A+\rho B)$, i.e. $R(A+\rho B)$ is closed.
Since $N(A+\rho B)=N(A)\cap N(B)=\set{0}$, this implies that $R(A+\rho B)=\HH$.
\end{proof}

\item To end this section we comment on the case in which $N(A)\cap N(B)\neq\{0\}$. Assume that $a\in R(A) + R(B)$ and $b\in R(B)$. Decomposing $x\in\HH$ as $x=x_0 + x_1$ with $x_0\in N(A)\cap N(B)$ and $x_1\in \HH':=\big(N(A)\cap N(B)\big)^\bot$, it is easy to see that
\begin{align*}
f(x) = f(x_0 + x_1)=f(x_0) + f(x_1)=f(x_1) \quad \text{and} \quad g(x) = g(x_0 + x_1)=g(x_0) + g(x_1)=g(x_1),
\end{align*}
because $f(x_0)=g(x_0)=0$. Hence, if $x\in \ZZ(\beta,a,b)$ then $x_1\in\ZZ(\beta,a,b)$, and also 
\[
x_1 + u\in \ZZ(\beta,a,b) \qquad \text{for every $u\in N(A)\cap N(B)$}.
\]
Note that $x_1$ is also a solution to Problem \ref{pb 2} if it is stated in the Hilbert space $\HH'$. Therefore, 
\[
\ZZ_\HH(\beta,a,b)= \ZZ_{\HH'}(\beta,a,b) + N(A)\cap N(B),
\]
where $\ZZ_{\HH}(\beta,a,b)$ and $\ZZ_{\HH'}(\beta,a,b)$ are the sets of solutions to Problem \ref{pb 2} in $\HH$ and in $\HH'$, respectively. 

\medskip

Now, assuming that $a\in\HH$ and $b\in\HH'$, we show that if $\ZZ(\beta,a,b)\neq\varnothing$ then $a\in\HH'$. Indeed, $x\in\ZZ(\beta,a,b)$ if and only if $g(x)=\beta$ and 
\[
(A+\la B)x=-(a+\la b) \qquad \text{for some $\la\in[\la_-,\la_+]$}.
\] 
Decomposing $x=x_0 + x_1$ and $a=a_0+a_1$  with $x_0,a_0\in N(A)\cap N(B)$ and $x_1,a_1\in \HH'$, we obtain
\[
(A+\la B)x_1=(A+\la B)x=-(a+\la b)=-(a_1+\la b) -a_0.
\]
Thus, $a_0=0$ and $a\in\HH'$, as we announced it. A similar argument shows that, if $\ZZ(\beta,a,b)\neq\varnothing$ for $a\in\HH'$ and $b\in\HH$, then $b\in\HH'$.

\medskip

Finally, let us mention that the situation in which $\ZZ(\beta,a,b)\neq\varnothing$ for both $a\notin \HH'$ and $b\notin\HH'$ is quite rare. In fact,
assume that $x\in \ZZ(\beta,a,b)$ and decompose $a$, $b$, and $x$ as $a=a_0+a_1$, $b=b_0+b_1$ and $x=x_0 + x_1$ with $a_0,b_0,x_0\in N(A)\cap N(B)$
and $a_1,b_1,x_1\in \HH'$. Then, by Proposition \ref{prop_unicidad_lambda}, there exists a unique $\la\in[\la_-,\la_+]$ such that $(A+\la B)x=-(a+\la b)$, and this condition can be written as
\[
(A+\la B)x_1=(A+\la B)x=-(a+\la b)=-(a_1+\la b_1) -(a_0+\la b_0).
\]
Hence, $a_0+\la b_0=0$. This is very restrictive due to the uniqueness of $\la$. 
 
\medskip

All these  remarks imply that the subspace $N(A)\cap N(B)$ plays no significant role in the study of Problem \ref{pb 2}. 

\end{enumerate}

\section{Appendix: Proofs of Propositions \ref{prop_unicidad_lambda} and \ref{prop_necesarias_suficientes_normal}}

For the sake of completeness we include proofs of Propositions \ref{prop_unicidad_lambda} and \ref{prop_necesarias_suficientes_normal}, which are similar to those
of some already published results.

\begin{customprop}{3.4}
Given $\beta\in\RR$, let $(a,b)\in \mc{A}^c$ be such that $\ZZ(\beta,a,b)\neq\varnothing$.
Then, there exists a unique $\lambda\in I_\geq(A,B)$ such that
$$(A+\lambda B) x =-(a+\lambda b),$$
for every $ x \in \ZZ(\beta,a,b)$.
\end{customprop}

\begin{proof}
Let $I_\geq(A,B)=[\lambda_-,\lambda_+]$. If $\lambda_-=\lambda_+$, then the result is trivial. Assume now that $\lambda_-\neq\lambda_+$. Given
$x_1,x_2\in\ZZ(\beta,a,b)$, let $\lambda_1,\lambda_2\in I_\geq(A,B)$ be such that
\begin{align}\label{eq:ecuaciones}
(A+\lambda_1 B)x_1&=-(a+\lambda_1 b),\nonumber\\
(A+\lambda_2 B)x_2&=-(a+\lambda_2 b).
\end{align}
Since $g(x_i)=\beta$ for $i=1,2$,
\begin{equation}\label{eq:valor_minimo}
f(x_i)=f(x_i)+\lambda_i\big(g(x_i)-\beta)=\PI{(A+\lambda_i B)x_i}{x_i}+2\real\PI{a+\lambda_i B}{x_i}-\la_i\beta=\PI{a+\lambda_i b}{x_i}-\lambda_i\beta.
\end{equation}
Then,
\begin{equation}\label{eq:ecuacion_1}
\PI{a+\lambda_1b}{x_1}-\lambda_1\beta=\PI{a+\lambda_2b}{x_2}-\lambda_2\beta,
\end{equation}
because $f(x_1)=f(x_2)$. Also,
\begin{align*}
\PI{(A+\lambda_1 B)x_1}{x_2}&=-\PI{a+\lambda_1b}{x_2}=-\PI{a+\lambda_2b}{x_2}+(\lambda_2-\lambda_1)\PI{b}{x_2},\\
\PI{(A+\lambda_2 B)x_2}{x_1}&=-\PI{a+\lambda_2b}{x_1}=-\PI{a+\lambda_1b}{x_1}+(\lambda_1-\lambda_2)\PI{b}{x_1},
\end{align*}
implies that
\begin{align}\label{eq:sistema}
\PI{Ax_1}{x_2}+\lambda_1\PI{Bx_1}{x_2}+(\lambda_1-\lambda_2)\PI{b}{x_2}+\lambda_2\beta=-\big(\PI{a+\lambda_2b}{x_2}-\lambda_2\beta\big),\nonumber\\
\PI{Ax_2}{x_1}+\lambda_2\PI{Bx_2}{x_1}+(\lambda_2-\lambda_1)\PI{b}{x_1}+\lambda_1\beta=-\big(\PI{a+\lambda_1b}{x_1}-\lambda_1\beta\big).
\end{align}
Combining \eqref{eq:ecuacion_1} and \eqref{eq:sistema} yields
\begin{align}\label{eq:ecuacion_2}
&\lambda_1\parentesisb{\PI{Bx_1}{x_2}+\PI{b}{x_1}+\PI{b}{x_2}-\beta}=\\\nonumber
&=\lambda_2\parentesisb{\PI{Bx_2}{x_1}+\PI{b}{x_1}+\PI{b}{x_2}-\beta}+\PI{Ax_2}{x_1}-\PI{Ax_1}{x_2}.
\end{align}
This implies that if $\real\PI{Bx_1}{x_2}+\real\PI{b}{x_1}+\real\PI{b}{x_2}-\beta\neq0$, then $\lambda_1=\lambda_2$.
Suppose now that 
\begin{equation}\label{eq:ecuacion_2}
\real\PI{Bx_1}{x_2}+\real\PI{b}{x_1}+\real\PI{b}{x_2}-\beta=0.
\end{equation}
By Cauchy-Schwarz inequality
\begin{equation}\label{eq:4}
\big(\real\PI{(A+\lambda_2 B)x_1}{x_2}\big)^2\leq|\PI{(A+\lambda_2B)x_1}{x_2}|^2\leq\PI{(A+\lambda_2B)x_1}{x_1}\PI{(A+\lambda_2B)x_2}{x_2}.
\end{equation}
We now proceed to derive expressions for the left and right sides of \eqref{eq:4}, by using the assumption \eqref{eq:ecuacion_2}.

In the first place, $\PI{a+\la_2b}{x_2}\in\RR$, and combining the first equation in \eqref{eq:sistema} and \eqref{eq:ecuacion_2}
we have that
\begin{equation}\label{eq:1}
\real\PI{(A+\lambda_2 B)x_1}{x_2}=-\PI{a+\lambda_2b}{x_2}+(\lambda_1-\lambda_2)\parentesis{\real\PI{b}{x_1}-\beta}.
\end{equation}
Secondly, by \eqref{eq:ecuacion_1},
\begin{align}\label{eq:2}
&\PI{(A+\lambda_2 B)x_1}{x_1}=\nonumber\\
&=\Big(-\PI{a+\lambda_2b}{x_2}+(\lambda_1-\lambda_2)\parentesis{\real\PI{b}{x_1}-\beta}\Big)-(\lambda_1-\lambda_2)\Big(\PI{Bx_1}{x_1}+\real\PI{b}{x_1}\Big)
\end{align}
Finally, considering that
\[
\PI{Bx_1}{x_1}+2\real\PI{b}{x_1}-\beta=g(x_1)-\beta=0,
\]
it follows
\begin{align}\label{eq:3}
&\PI{(A+\lambda_2B)x_2}{x_2}=\nonumber\\
&=\Big(-\PI{a+\lambda_2b}{x_2}+(\lambda_1-\lambda_2)\parentesis{\real\PI{b}{x_1}-\beta}\Big)+(\lambda_1-\lambda_2)\Big(\PI{Bx_1}{x_1}+\real\PI{b}{x_1}\Big).
\end{align}
Then, replacing \eqref{eq:1}, \eqref{eq:2} and \eqref{eq:3} in \eqref{eq:4} yields
\begin{align}\label{eq:horrible_1}
&\Big(-\PI{a+\lambda_2b}{x_2}+(\lambda_1-\lambda_2)\parentesis{\real\PI{b}{x_1}-\beta}\Big)^2\leq\nonumber\\
&\leq\Big(-\PI{a+\lambda_2b}{x_2}+(\lambda_1-\lambda_2)\parentesis{\real\PI{b}{x_1}-\beta}\Big)^2-\bigg((\lambda_1-\lambda_2)\Big(\PI{Bx_1}{x_1}+\real\PI{b}{x_1}\Big)\bigg)^2.
\end{align}
Analogously, considering the inequality
\begin{equation*}
\big(\real\PI{(A+\lambda_1 B)x_1}{x_2}\big)^2\leq|\PI{(A+\lambda_1B)x_1}{x_2}|^2\leq\PI{(A+\lambda_1B)x_1}{x_1}\PI{(A+\lambda_1B)x_2}{x_2},
\end{equation*}
and following the procedure above but interchanging $x_1$ with $x_2$ and $\lambda_1$ with $\lambda_2$ yields
\begin{align}\label{eq:horrible_2}
&\Big(-\PI{a+\lambda_1b}{x_1}+(\lambda_2-\lambda_1)\parentesis{\real\PI{b}{x_2}-\beta}\Big)^2\leq\nonumber\\
&\leq\Big(-\PI{a+\lambda_1b}{x_1}+(\lambda_2-\lambda_1)\parentesis{\real\PI{b}{x_2}-\beta}\Big)^2-\bigg((\lambda_2-\lambda_1)\Big(\PI{Bx_2}{x_2}+\real\PI{b}{x_2}\Big)\bigg)^2.
\end{align}
Hence, from \eqref{eq:horrible_1} and \eqref{eq:horrible_2}, if $\PI{Bx_1}{x_1}+\real\PI{b}{x_1}\neq0$ or $\PI{Bx_2}{x_2}+\real\PI{b}{x_2}\neq0$, then $\lambda_1=\lambda_2$.
Assume that
\[
\PI{Bx_1}{x_1}+\real\PI{b}{x_1}=0\qquad\text{and}\qquad\PI{Bx_2}{x_2}+\real\PI{b}{x_2}=0.
\]
Since $g(x_1)=\beta$ and $g(x_2)=\beta$, this implies that
\[
\PI{Bx_1}{x_1}=-\real\PI{b}{x_1}=-\beta\qquad\text{and}\qquad\PI{Bx_2}{x_2}=-\real\PI{b}{x_2}=-\beta,
\]
which in turn, by \eqref{eq:ecuacion_1}, implies that
\[
\real\PI{a}{x_1}=\real\PI{a}{x_2}.
\]
Since $f(x_1)=f(x_2)$, from this last equation it follows that $\PI{Ax_1}{x_1}=\PI{Ax_2}{x_2}$. From \eqref{eq:ecuacion_2} we also have
that $\real\PI{Bx_1}{x_2}=-\beta$. Finally,
\[
\PI{Ax_1}{x_1}=\PI{(A+\lambda_1B)x_1}{x_1}-\lambda_1\PI{Bx_1}{x_1}=-\PI{a+\lambda_1b}{x_1}-\lambda_1\PI{Bx_1}{x_1}=-\real\PI{a}{x_1},
\]
and by \eqref{eq:1},
\begin{align*}
\real\PI{Ax_1}{x_2}&=\real\PI{(A+\lambda_2 B)x_1}{x_2}-\lambda_2\real\PI{Bx_1}{x_2}\\
&=-\PI{a+\lambda_2b}{x_2}+(\lambda_1-\lambda_2)\parentesis{\real\PI{b}{x_1}-\beta}-\lambda_2\real\PI{Bx_1}{x_2}\\
&=-\real\PI{a}{x_2}.
\end{align*}
Hence,
\[
\real\PI{Ax_1}{x_2}=\PI{Ax_1}{x_1}=\PI{Ax_2}{x_2}.
\]
Let $\rho\in(\lambda_-,\lambda_+)$, then
\begin{align*}
&\PI{(A+\rho B)(x_1-x_2)}{x_1-x_2}=\\
&=\big(\PI{Ax_1}{x_1}+\PI{Ax_2}{x_2}-2\real\PI{Ax_1}{x_2}\big)+\rho\big(\PI{Bx_1}{x_1}+\PI{Bx_2}{x_2}-2\real\PI{Bx_1}{x_2}\big)\\
&=0.
\end{align*}
Then, by Proposition \ref{nucleo del interior},
\[
x_1-x_2\in N\big((A+\rho B)^{1/2}\big)= N(A+\rho B)=N(A)\cap N(B)=\set{0},
\]
and thus $x_1=x_2$. Then,
\[
-(a+\lambda_1b)=(A+\lambda_1B)x_1=(A+\lambda_1B)x_2=(A+\lambda_2 B)x_2+(\lambda_1-\lambda_2)Bx_2=-(a+\lambda_2 b)+(\lambda_1-\lambda_2)Bx_2,
\]
and consequently
\[
(\lambda_1-\lambda_2)(Bx_2+b)=0.
\]
By \eqref{eq:ecuaciones}, if $Bx_2=-b$ then $Ax_2=-a$ which is a contradiction, because $(a,b)\in\mc{A}^c$.
Hence, $\lambda_1=\lambda_2$.
\end{proof}

Define
\[
\mc{D_+}:=\HH_+\ominus \N_+\qquad\text{and}\qquad\mc{D_-}:=\HH_-\ominus \N_-,
\]
the \emph{positive defect subspace} of $\N_+$ and \emph{negative defect subspace} of $\N_-$, respectively.
For the proof of Proposition \ref{prop_necesarias_suficientes_normal} we make use of the next two technical lemmas. Their proofs follow the lines of the
proofs of \cite[Lemma 4.5]{GZMMP} and \cite[Lemma 4.6]{GZMMP}, respectively.

\begin{lema*}\label{lema:primera_serie}
Given $u\in\HH_\pm$, decompose it as $u=v+w$ with $v\in\mc{N}_\pm$ and $w\in\mc{D}_\pm$.
Then, for every $\tau\in(-\kappa,\kappa)$,
\begin{align*}
\|G_\pm^{1/2}(I_\pm\pm \tau G_\pm)^{-1}u\|^2 
=\frac{\kappa}{(\kappa\pm \tau)^2}\|v\|^2+\|G_\pm^{1/2}(I_\pm\pm \tau G_\pm)^{-1}w\|^2.
\end{align*}
\end{lema*}

\begin{lema*}\label{lema:en_el_rango} 
Given $u\in\HH_\pm$, if $\lim_{\tau\to\kappa}\|G_\pm^{1/2}(I_\pm-\tau G_\pm)^{-1}u\|<+\infty$,
then $u\in R(I_\pm-\kappa G_\pm)$.
\end{lema*}

\begin{customprop}{4.9} For every $u_0\in\HH$ and every $\nu\in\RR$ there exists $(\gamma,y)\in[-\kappa,\kappa]\x\HH$
satisfying
\[
\begin{cases}
\,(I+\gamma G)y=u_0,\\
\,\PI{Gy}{y}=\nu,
\end{cases}
\]
if and only if $\mc{N}_-\neq\set{0}$ and $\mc{N}_+\neq\set{0}$.
\end{customprop}

\begin{proof} The necessity follows by \cite[Lemma 4.7]{GZMMP}. Conversely, assume that $\mc{N}_-\neq\set{0}$ and $\mc{N}_+\neq\set{0}$.
Consider the decomposition $u_0=u_0^++u_0^-+u_0^0$ with $u_0^\pm\in\HH_\pm$ and $u_0^0\in N(G)$, and the real valued functions $f_+$ and $f_-$ defined by
\begin{equation}\label{eq:funcion_+}
f_+(\tau)=\|G_+^{1/2}(I_++\tau G_+)^{-1}u_0^+\|^2,\quad \tau\in(-\kappa,\kappa]
\end{equation}
and
\begin{equation}\label{eq:funcion_-}
f_-(\tau)=\|G_-^{1/2}(I_--\tau G_-)^{-1}u_0^-\|^2,\quad \tau\in[-\kappa,\kappa),
\end{equation}
respectively. If there exists $\tau_0\in(-\kappa,\kappa)$ such that $f_+(\tau_0)=f_-(\tau_0)+\nu$, then setting $\gamma=\tau_0$ and
\[
y=(I_++\gamma G_+)^{-1}u_0^++(I_--\gamma G_-)^{-1}u_0^-+u_0^0
\]
implies that $\PI{Gy}{y}=\nu$ and $(I+\gamma G)y=u_0$.

Now assume that $f_+(\tau)>f_-(\tau)+\nu$ for every $\tau\in(-\kappa,\kappa)$. Then,
\[
\lim_{\tau\to\kappa}\|G_-^{1/2}(I_--\tau G_-)^{-1}u_0^-\|^2=\lim_{\tau\to\kappa} f_-(\tau)\leq f_+(\kappa)-\nu.
\]
Thus 
$u_0^-\in R(I_--\kappa G_-)$, and consequently $u_0\in R(I+\kappa G)$.
Since $\N_-\neq\set{0}$, we can choose $z\in \N_-$ with $\|z\|=1$.
Considering that $G_-z=\tfrac{1}{\kappa}z$ and $(I_--\kappa G_-)^\dag u_0^-\bot z$, and setting
\[
y=(I_++\kappa G_+)^{-1}u_0^++(I_--\kappa G_-)^\dag u_0^-+u_0^0+\alpha_-z,
\]
with
\begin{equation*}
\alpha_-:=\bigg(\kappa\Big(\|G_+^{1/2}(I_++\kappa G_+)^{-1}u_0^+\|^2-\|G_-^{1/2}(I_--\kappa G_-)^\dag u_0^-\|^2-\nu\Big)\bigg)^{1/2},
\end{equation*}
we have that $\PI{Gy}{y}=\nu$ and $(I+\kappa G)y=u_0$.

A similar argument holds if we assume that $f_+(\tau)<f_-(\tau)+\nu$ for every $\tau\in(-\kappa,\kappa)$.
\end{proof}

\end{document}